\documentclass[mathpazo]{cicp}

\usepackage{bm}
\usepackage{mathdots}
\usepackage{extarrows}
\allowdisplaybreaks
\usepackage[top=1in,bottom=1in,left=1.25in,right=1.25in]{geometry}
\usepackage{amsfonts}
\usepackage{amsthm}
\usepackage{amssymb,indentfirst}
\usepackage{amsmath}
\usepackage[divps]{xcolor}

\usepackage{mathrsfs}
\usepackage{cases}
\usepackage{titletoc}
\usepackage{multicol}

\usepackage{amssymb}
\newtheorem{thm}{Theorem}[section]

\newtheorem{lem}[thm]{Lemma}

\newtheorem{defn}[thm]{Definition}

\newtheorem{ex}[thm]{Example}

\def\bc{\begin{center}}       \def\ec{\end{center}}
\def\be{\begin{equation}}     \def\ee{\end{equation}}
\def\ba{\begin{array}}        \def\ea{\end{array}}
\def\bea{\begin{eqnarray}}    \def\eea{\end{eqnarray}}
\def\beaa{\begin{eqnarray*}}  \def\eeaa{\end{eqnarray*}}

\def\ifl{\iffalse}

\begin{document}
\title[Series solutions to the cauchy problem for partial differential equations]{Series solutions to the cauchy problem for partial differential equations}


\author[T. Zhang, Alatancang]{Tao Zhang$^{a}$, \quad  Alatancang$^{b,}$\corrauth}

\address{ $^a$School of Mathematical Sciences of Inner Mongolia University, Hohhot, 010021, China\\
           $^b$Huhhot University for Nationalities, Hohhot, 010051, China}

\email{{\tt zhangtaocx@163.com (T. Zhang), alatanca@imu.edu.cn (A. Chen).}}



\begin{abstract}
The method of separation of variables can be used to solve many separable linear partial differential equations (LPDEs). Moreover, variable separation solutions usually are some trigonometric series.
In the paper, base on some ideas of this method, we introduce a new technique to solve the Cauchy problem for some LPDEs with the initial conditions consisting of some trigonometric series, power series and exponential series. Then many LPDEs which are not separable are solved, such as some second order elliptic equations, Stokes equations and so on. In addition, the solutions of them can be expressed by trigonometric series, power series or exponential series. Moreover, by using power amd exponential series and an iterative method, we can solve many LPDEs and nonlinear PDEs for the first time.
\end{abstract}


\ams{35A09, 35C10, 76D07}

\keywords{partial differential equation;   series solution; Cauchy problem; separation of variables; Stokes equations}

\maketitle


\section{Introduction}\label{intr}

  The method of separation of variables (also known as the Fourier method) is one of the oldest and most widely used techniques for solving linear partial differential equations (LPDEs) \cite{ht}. 
  Moreover, variable separation solutions are trigonometric series in general. However, this method can not be used to deal with the LPDEs which are not separable. In addition, variable separation solutions usually could not be expressed by power series and exponential series.

 For simplicity, we write
        \begin{equation*}
            \begin{array}{l}
               T_{1}=\sum_{j=1}^{n}a_{j}\frac{\text{d}^{2j}}{\text{d}x^{2j}};\\
               T_{2}=\sum_{j=1}^{n}a_{j}\frac{\text{d}^{j}}{\text{d}x^{j}};\\
               T_{3}=\sum_{j=1}^{n}a_{j}x^j\frac{\text{d}^{j}}{\text{d}x^{j}}.
             \end{array}
        \end{equation*}

 Some keys in solving many LPDEs (such as the wave equation) by using the method of separation of variables are:

 (A) The functions in the boundary conditions of these LPDEs can be expressed by trigonometric series (Fourier series).

 (B) Stretching transformation (eigenvalues and eigenvectors):
 \begin{equation*}
   \begin{array}{l}
      T_1 \sin kx=\left(\sum_{j=1}^{n}(-1)^ja_{j}k^{2j}\right) \sin kx,\ \ k=1,2,\cdots;\\
      T_1 \cos kx=\left(\sum_{j=1}^{n}(-1)^ja_{j}k^{2j}\right) \cos kx,\ \ k=0,1,2,\cdots;\\
      T_2 e^{ ikx}=\left(\sum_{j=1}^{n}a_{j}(ik)^{j}\right)e^{ ikx},\ \ k=0,\pm1,\pm2,\cdots.
   \end{array}
 \end{equation*}

  Our work's motivation partly comes from the following ideas:

  (A') There are some functions which can be expressed by power series or exponential series (Taylor series). For examples:
  \begin{equation*}
  \begin{array}{l}
     \frac{1}{1-x}=\sum_{k=0}^{+\infty}x^k,\ \ x\in (-1,1);\\
     \frac{1}{1-e^x}=\sum_{k=0}^{+\infty}e^{kx},\ \ x\in (-\infty,0).
  \end{array}
  \end{equation*}

 (B') Stretching transformation (eigenvalues and eigenvectors):
 \begin{equation*}
   \begin{array}{l}
      T_2 e^{ kx}=\left(\sum_{j=1}^{n}a_{j}k^{j}\right)e^{ kx},\ \  k=0,1,2,\cdots;\\
      T_3 x^k=\left(\sum_{j=1}^{n}a_{j}\prod_{m=0}^{j-1}(k-m)\right)x^k,\ \  k=0,1,2,\cdots.
   \end{array}
 \end{equation*}

      Then base on (A), (B), (A'), (B'), the undetermined coefficient method and the superposition principle for the solution of LPDEs, we introduce a new technique, by using which we can solve the cauchy problem for many LPDEs even if they are not separable, such as some second order elliptic equation, Stokes equations and so on. Moreover, the solutions of them can be trigonometric series, power series or exponential series.

      Let $\Lambda_1=\{e^{\lambda kx}\}_{k=0}^{+\infty},\ \Lambda_2=\{x^{\mu k}\}_{k=0}^{+\infty}$ where $\lambda ,\mu\in \mathbb{R}\setminus \{0\}$.
    Then for any $m_1,m_2=0,1,2,\cdots$, we have
    \begin{equation}\label{1.bf}
    \left\{
      \begin{array}{l}
         e^{\lambda m_1x}e^{\lambda m_2x}=e^{\lambda (m_1+m_2)x}\in \Lambda_1,\\
         x^{\mu m_1}x^{\mu m_2}=x^{\mu (m_1+m_2)}\in \Lambda_2,\\
         m_1+m_2\geq\max\{m_1,m_2\}.
       \end{array}
       \right.
    \end{equation}
  Then by using an iterative method with respect to \eqref{1.bf}, we can solve many LPDEs and nonlinear PDEs (NPDEs) for the first time.


 \section{Series solutions to the cauchy problem for some LPDEs}

Notation

     \begin{equation*}
      \begin{array}{ll}
       \mathbb{R}-\text{the  real  numbers}.\\
       \mathbb{C}-\text{the  complex numbers}.\\
       e^f=\exp(f).\\
       \mathbb{R}^{n}=\{(r_1,\cdots,r_n)\mid r_j\in \mathbb{R},\ 1\leq j\leq n\}.\\
       \mathbb{Z}^{n}=\{(k_{1},\cdots,k_{n})\mid k_{j}=0,\pm 1,\pm 2,\cdots,\ 1\leq j\leq n\}.\\
       \mathbb{N}^{n}=\{(k_{1},\cdots,k_{n})\mid k_{j}=0,1,2,\cdots,\ 1\leq j\leq n\}.\\
       \mathbb{N}_+^{n}=\{(k_{1},\cdots,k_{n})\mid k_{j}=1,2,\cdots,\ 1\leq j\leq n\}.\\
        \sum\limits_{k=(k_{1},\cdots,k_{n})\in \mathbb{Z}^{n}}a_{k}=\sum\limits_{m=0}^{+\infty}\ \sum\limits_{\mid k\mid= m}a_{k}
        ,\ \ \ \ \ \ \ \  |k|=\sum\limits_{j=1}^n|k_{j}|.\\
      \end{array}
      \end{equation*}

     Let $\Gamma=(\Gamma_{pq})_{n\times n}$ be an $n\times n$ matrix differential operator, and let
     \begin{equation*}
       \Gamma_{pq}=\sum\limits_{h=0}^{m_{pq}}A_{pqh}(t)\partial_{t}^{h}\sum\limits_{j=1}^{w_{pqh}}B_{pqhj}(x)\partial _{x}^{\alpha_{pqhj}},
     \end{equation*}
     where $1\leq p,q\leq n$, $m_{pq}\in \mathbb{N}$, $w_{pqh}\in \mathbb{N}_+$,
     $\alpha_{pqhj}\in \mathbb{N}^n$, $x\in \Omega\subseteq \mathbb{R}^{n}$, $A_{pqh}(t)\in C[0,T],$ $ B_{pqhj}(x)\in C(\Omega)$.
     In this section, we consider the Cauchy problem for the following LPDEs:
       \begin{numcases}{}
      \Gamma u(x,t)=f(x,t),\ \ x\in \Omega,\ 0\leq t\leq T,\label{1b}\\
         \partial_{t}^{h}u_q|_{t=0}=\sum\limits_{k\in \Lambda}r_{qhk}\xi_{k}\in C(\Omega),\ \
         1\leq q\leq n,\ 0\leq h\leq m_q-1, \label{3b}\\
       f_j=\sum\limits_{k\in \Lambda}\xi_{k}Z_{kj}(t)\in C\left(\Omega\oplus[0,T]\right),\ \ 1\leq j\leq n,\label{4b}
    \end{numcases}
    where $u=(u_1,\cdots,u_n)^T,$ $ f=(f_1,\cdots,f_n)^T$, $\xi_{k}\in C(\Omega)$, $k\in\bigwedge\subseteq \mathbb{N}^n$,
    and $m_q=\max\limits_{1\leq p\leq n}m_{pq},\ 1\leq q\leq n$.

     \begin{defn} We say Eq. \eqref{1b}-\eqref{4b} fulfils the Fourier-Taylor  conditions,
     which we shall denote by $ u(x,t)\in FT(\Omega\oplus[0,T]),\ \{\xi_k\}_{k\in \Lambda},$
     if for any $1\leq p,q\leq n,\ 0\leq h\leq m_{pq},\ 1\leq j\leq w_{pqh}$, there exists a sequence $\{l_{pqhjk}\}_{k\in \Lambda}\subseteq \mathbb{C}$ such that
      \begin{equation*}
          B_{pqhj}(x)D^{\alpha_{pqhj}}\xi_{k}=l_{pqhjk}\xi_{k},\ \ k\in \Lambda.
      \end{equation*}
     \end{defn}

         Next we solve Eq. \eqref{1b}-\eqref{4b} when $u(x,t)\in FT(U_{n,t}),\ \{\xi_k\}_{k\in \Lambda}$ holds.
      We let
    \begin{equation}\label{b1}
        u(x,t)=\sum\limits_{k\in \Lambda}\xi_{k}T_{k}(t),
    \end{equation}
    where $T_{k}(t)=(T_{k1}(t),\cdots,T_{kn}(t))^T$.
     Suppose that the following conditions hold:
     \begin{equation}\label{4oooo}
       \left\{
       \begin{array}{l}
       u=\sum\limits_{k\in \Lambda}\xi_{k}T_{k}(t)\in C(\Omega\oplus[0,T]),\\
       \\
          \partial_{t}^{h}\partial _{x}^{\alpha_{pqhj}}u
        =\sum\limits_{k\in \Lambda}T_{k}^{(h)}(t)\partial _{x}^{\alpha_{pqhj}}\xi_{k}\in C(\Omega\oplus[0,T]),\ \ 1\leq p,q\leq n,\ 0\leq h\leq m_{pq},\ 1\leq j\leq w_{pqh}.
       \end{array}
       \right.
     \end{equation}
    Then by substituting the series \eqref{4oooo} into Eq. \eqref{1b}-\eqref{4b} we have
   \begin{equation*}
   \left\{
     \begin{array}{l}
    \sum\limits_{k\in \Lambda}\xi_{k}\left(
   \sum\limits_{1\leq q\leq n,\ 0\leq h\leq m_{pq},\ 1\leq j\leq w_{pqh}}l_{pqhjk}A_{pqh}(t)T_{kq}^{(h)}(t)-Z_{kp}(t)\right)=0,\ \
    1\leq p\leq n,\\
    \partial_{t}^{h}u_q|_{t=0}=\sum\limits_{k\in \Lambda}T^{(h)}_{kq}(0)\xi_{k}=\sum\limits_{k\in \Lambda}r_{qhk}\xi_{k},\ \
         1\leq q\leq n,\ 0\leq h\leq m_q-1,
     \end{array}
   \right.
   \end{equation*}
  For every $k\in \Lambda$, let
 \begin{equation}\label{fsa}
   \left\{
     \begin{array}{c}
   \sum\limits_{1\leq q\leq n,\ 0\leq h\leq m_{pq},\ 1\leq j\leq w_{pqh}}l_{pqhjk}A_{pqh}(t)T_{kq}^{(h)}(t)-Z_{kp}(t)=0,\ \
    1\leq p\leq n,\\
   T^{(h)}_{kq}(0)=r_{qhk},\ \ \ \ \ \ 1\leq q\leq n,\ 0\leq h\leq m_q-1.
     \end{array}
   \right.
   \end{equation}
    This is a Cauchy problem for an ODEs, so we may get $T_{kq}(t),\ 1\leq q\leq n,\ k\in \Lambda$.

 \begin{defn}We call the series \eqref{b1} a formal solution of Eq. \eqref{1b}-\eqref{4b} w.r.t. (with respect to) $\{\xi_{k}\}_{k\in \Lambda}$.
     \end{defn}

     \begin{thm} If $u(x,t)\in FT(\Omega\oplus[0,T]),\ \{\xi_k\}_{k\in \Lambda}$,
     and if the solution of Eq. \eqref{fsa}
    exists and is unique for every $k\in \Lambda$ , then the formal solution of Eq. \eqref{1b}-\eqref{4b} w.r.t.
     $\{\xi_{k}\}_{k\in \Lambda}$ exists and is unique.
     \end{thm}

     \begin{thm}\label{thm3.3} Suppose that the series \eqref{b1} is a formal solution of Eq. \eqref{1b}-\eqref{4b} w.r.t.
    $\{\xi_{k}\}_{k\in \Lambda}$. If it satisfies the conditions \eqref{4oooo}, then it is a solution of Eq. \eqref{1b}-\eqref{4b}.
    \end{thm}

    Clearly if $\Lambda$ is a finite set, then the conditions \eqref{4oooo} hold. So we have:

    \begin{thm}\label{thm3.4} If $\Lambda$ is a finite set, then a formal solution of Eq. \eqref{1b}-\eqref{4b} w.r.t.
    $\{\xi_{k}\}_{k\in \Lambda}$ is a solution.
    \end{thm}


 Next we solve a well known partial differential equation by using the above technique. The result we obtain is exactly the same as the variable separation solutions. However, our technique is more simple and intuitive.

   \begin{ex} (Wave Equation \cite{WE}).
     \begin{numcases}{}
           u_{tt}-a^2u_{xx}=0,\ \
            0\leq x\leq l,\ t\geq 0,\ a\in \mathbb{R}\setminus \{0\},\label{3.1a1} \\
           u(x,0)=\sum\limits_{k\in \mathbb{N}_+}A_k\sin\frac{k\pi x}{l},\ \ \ u_t(x,0)=\sum\limits_{k\in \mathbb{N}_+}B_k\sin\frac{k\pi x}{l},\label{3.2a1}\\
          u(0,t)=u(l,t)=0.\label{3.aaa}
    \end{numcases}
    \end{ex}

   If we delete the condition \eqref{3.aaa}, then $u(x,t)\in FT([0,l]\oplus [0,+\infty))$, $\left\{\sin\frac{k\pi x}{l}\right\}_{k\in \mathbb{N}_+}$. Next we set
    \begin{equation}\label{3.a3}
        u(x,t)=\sum\limits_{k\in \mathbb{N}_+}T_{k}(t)\sin\frac{k\pi x}{l}.
    \end{equation}
    It satisfies \eqref{3.aaa}.
     Suppose that the series \eqref{3.a3} satisfies  the following conditions:
     \begin{equation}\label{3.aa}
     \left\{
    \begin{array}{l}
    u=\sum\limits_{k\in \mathbb{N}_+}T_{k}(t)\sin\frac{k\pi x}{l}\in C([0,l]\oplus [0,+\infty)),\\
    u_{xx}=\sum\limits_{k\in \mathbb{N}_+}T_{k}(t)\left(\sin\frac{k\pi x}{l}\right)''\in C([0,l]\oplus [0,+\infty)),\\
    u_{tt}=\sum\limits_{k\in \mathbb{N}_+}T''_{k}(t)\sin\frac{k\pi x}{l}\in C([0,l]\oplus [0,+\infty)).
   \end{array}
   \right.
   \end{equation}
     Then by substituting the series \eqref{3.aa} into Eq. \eqref{3.1a1}-\eqref{3.aaa} we have
   \begin{equation*}
    \left\{
      \begin{array}{ll}
       \sum\limits_{k\in \mathbb{N}_+}\left(T_{k}''+\left(\frac{ak\pi }{l}\right)^2T_{k}\right)\sin\frac{k\pi x}{l}=0,\\
        u(x,0)=\sum\limits_{k\in \mathbb{N}_+}T_{k}(0)\sin\frac{k\pi x}{l}=\sum\limits_{k\in \mathbb{N}_+}A_k\sin\frac{k\pi x}{l},\\
   u_t(x,0)=\sum\limits_{k\in \mathbb{N}_+}T'_{k}(0)\sin\frac{k\pi x}{l}=\sum\limits_{k\in \mathbb{N}_+}B_k\sin\frac{k\pi x}{l}.
      \end{array}
    \right.
   \end{equation*}
    Next for any $k\in \mathbb{N}_+$, we let
 \begin{equation*}
    \left\{
      \begin{array}{ll}
      T_{k}''+\left(\frac{ak\pi }{l}\right)^2T_{k}=0,\\
      T_{k}(0)= A_k,\\
      T'_{k}(0)= B_k.
      \end{array}
    \right.
   \end{equation*}
    Then we have
    \begin{equation*}\label{3.a7}
    \begin{array}{cc}
        T_{k}(t)=A_{k}\cos\frac{ak\pi }{l}t+\frac{l}{ak\pi }B_{k}\sin\frac{ak\pi}{l}t,\ \ \ \ k\in \mathbb{N}_+.
        \end{array}
    \end{equation*}
  So the formal solution of Eq. \eqref{3.1a1}-\eqref{3.aaa} w.r.t. $\left\{\sin\frac{k\pi x}{l}\right\}_{k\in \mathbb{N}_+}$ is:
    \begin{equation}\label{3.a9}
    \begin{array}{cc}
     u(x,t)=\sum\limits_{k\in \mathbb{N}_{+}}\left(A_{k}\cos\frac{ak\pi }{l}t+\frac{l}{ak\pi }B_{k}\sin\frac{ak\pi}{l}t\right)\sin\frac{k\pi x}{l},
      \end{array}
    \end{equation}

    \begin{thm}\label{thm3.7} If
    \begin{equation}\label{ieq1}
    \begin{array}{cc}
      \sum\limits_{k\in \mathbb{N}_+}k^2|A_k|+k|B_k|<+\infty,
      \end{array}
    \end{equation}
    then the series \eqref{3.a9} is a solution of Eq. \eqref{3.1a1}-\eqref{3.aaa}.
    \end{thm}

   \begin{proof} The inequality \eqref{ieq1} implies that the series \eqref{3.a9} satisfies the conditions \eqref{3.aa}, so it
   is a solution of Eq. \eqref{3.1a1}-\eqref{3.aaa} by Theorem \ref{thm3.3}.
   \end{proof}\vskip8pt

The following PDEs could not be dealt with by using the method of separation of variables, one reason is that they are not separable. However,
   we can get the exact solutions of them.

    \begin{ex} (Second order hyperbolic equation)
   \begin{equation}\label{4a.1}
        \left\{
          \begin{array}{ll}
           u_{tt}+au_{xt}+bu_{xx}=0,\\
           a,b\in \mathbb{R},\ \Delta= a^{2}-4b>0,\ x\in  \mathbb{R}, \ t\geq 0,\\
           u(x,0)=\sum\limits_{k\in \mathbb{N}}A_k\cos\frac{k\pi x}{l},\ u_t(x,0)=\sum\limits_{k\in \mathbb{N}_+}B_k\sin\frac{k\pi x}{l}.
          \end{array}
    \right.
    \end{equation}
    \end{ex}

   We let
 \begin{equation*}
    A'_k=\left\{
          \begin{array}{ll}
              \frac{A_{k}}{2},\ k\in \mathbb{N}_{+};\\
              A_{0},\ k=0;\\
           \frac{A_{-k}}{2},\ -k\in \mathbb{N}_{+};
          \end{array}
    \right.\
    B'_k=\left\{
          \begin{array}{ll}
              \frac{B_{k}}{2i},\ k\in \mathbb{N}_{+};\\
           \frac{-B_{-k}}{2i},\ -k\in \mathbb{N}_{+}.
          \end{array}
    \right.
 \end{equation*}
   Then we have
   \begin{equation*}\label{4a.2}
    \left\{
          \begin{array}{ll}
          u(x,0)=\sum\limits_{k\in \mathbb{N}}A_k\cos\frac{k\pi x}{l}=\sum\limits_{k\in \mathbb{Z}}A'_ke^{\frac{ik\pi x}{l}};\\
           u_t(x,0)=\sum\limits_{k\in \mathbb{N}_+}B_k\sin\frac{k\pi x}{l}=\sum\limits_{k\in \mathbb{Z}\setminus\{0\}}B'_ke^{\frac{ik\pi x}{l}}.
          \end{array}
    \right.
   \end{equation*}
    Thus $u(x,t)\in FT(\mathbb{R}\oplus[0,\infty)),\ \left\{e^{\frac{ik\pi x}{l}}\right\}_{k\in \mathbb{Z}}$. So we let
    \begin{equation}\label{4a.3}
        u(x,t)=\sum\limits_{k\in \mathbb{Z}}T_{k}(t)e^{\frac{ik\pi x}{l}}.
    \end{equation}
      Suppose that the series \eqref{4a.3} satisfies  the following conditions:
     \begin{equation}\label{4.aa}
     \left\{
    \begin{array}{l}
      \sum\limits_{k\in \mathbb{Z}}T_{k}(t)e^{\frac{ik\pi x}{l}}\in C(\mathbb{R}\oplus[0,\infty)),\\
    \frac{\partial^2}{\partial x^2}\sum\limits_{k\in \mathbb{Z}}T_{k}(t)e^{\frac{ik\pi x}{l}}
    =\sum\limits_{k\in \mathbb{Z}}T_{k}(t)\left(e^{\frac{ik\pi x}{l}}\right)''\in C(\mathbb{R}\oplus[0,\infty)),\\
    \frac{\partial^2}{\partial x\partial t}\sum\limits_{k\in \mathbb{Z}}T_{k}(t)e^{\frac{ik\pi x}{l}}
    =\sum\limits_{k\in \mathbb{Z}}T'_{k}(t)\left(e^{\frac{ik\pi x}{l}}\right)'\in C(\mathbb{R}\oplus[0,\infty)),\\
     \frac{\partial^2}{\partial t^2}\sum\limits_{k\in \mathbb{Z}}T_{k}(t)e^{\frac{ik\pi x}{l}}
     =\sum\limits_{k\in \mathbb{Z}}T''_{k}(t)e^{\frac{ik\pi x}{l}}\in C(\mathbb{R}\oplus[0,\infty)).
   \end{array}
   \right.
   \end{equation}
    Then by substituting the series \eqref{4.aa} into Eq. \eqref{4a.1} we have
   \begin{equation*}\label{4a.4}
\left\{
   \begin{array}{l}
\sum\limits_{k\in \mathbb{Z}}\left(T_{k}''+\frac{ik\pi }{l}aT_{k}'-(\frac{k\pi }{l})^2bT_{k}\right)e^{\frac{ik\pi x}{l}}=0,\\
    u(x,0)=\sum\limits_{k\in \mathbb{Z}}T_{k}(0)e^{\frac{ik\pi x}{l}}=\sum\limits_{k\in \mathbb{Z}}A'_{k}e^{\frac{ik\pi x}{l}},\\
 u_t(x,0)=\sum\limits_{k\in \mathbb{Z}}T'_{k}(0)e^{\frac{ik\pi x}{l}}=\sum\limits_{k\in \mathbb{Z}\setminus \{0\}}B'_{k}e^{\frac{ik\pi x}{l}}.
   \end{array}
\right.
   \end{equation*}
  For any $k\in \mathbb{Z}$, let
  \begin{equation*}
  \left\{
  \begin{array}{l}
    T_{k}''+\frac{ik\pi }{l}a T_{k}'-\left(\frac{k\pi }{l}\right)^2bT_{k}=0, \\
   T_{k}(0)=A'_{k},\\
   T'_{k}(0)=B'_{k},
  \end{array}
  \right.
  \end{equation*}
   where $B'_0=0$. Then we  get the formal solution of Eq. \eqref{4a.1} w.r.t.
    $\left\{e^{\frac{ik\pi x}{l}}\right\}_{k\in \mathbb{Z}}$:
    \begin{equation}\label{4a.8}
      u=A_0+\sum\limits_{k\in  \mathbb{N}_+}\left[\frac{\left(\sqrt{\Delta}+a\right)k\pi A_k-2lB_{k}}{2k\pi \sqrt{\Delta}}
       \cos h_{1k}
       +\frac{\left(\sqrt{\Delta}-a\right)k\pi A_k+2lB_{k}}{2k\pi \sqrt{\Delta}}\cos h_{2k}\right],
    \end{equation}
 where
    \begin{equation*}
          h_{1k}=\frac{-a+\sqrt{\Delta}}{2}\frac{k\pi }{l}t+\frac{k\pi }{l}x,\ \ h_{2k}=\frac{-a-\sqrt{\Delta}}{2}\frac{k\pi }{l}t+\frac{k\pi }{l}x,
          \ \ \ \ k\in \mathbb{N}_+.
    \end{equation*}\vskip8pt

Similar as Theorem \ref{thm3.7}, we have

  \begin{thm} If
   \begin{equation*}
    \begin{array}{cc}
      \sum\limits_{k\in \mathbb{N}_+}k^2|A_k|+k|B_k|<+\infty,
      \end{array}
    \end{equation*}
    then the series \eqref{4a.8} is a solution of Eq. \eqref{4a.1}.
    \end{thm}

     \begin{ex} (Second order elliptic equation)
   \begin{equation}\label{4ae}
        \left\{
          \begin{array}{ll}
           u_{tt}+au_{xt}+bu_{xx}=0,\ \
           a,b>0,\ \Delta= a^{2}-4b<0,\ x\in  \mathbb{R}, \ t\geq 0,\\
            u(x,0)=\cos e^{2x},\ u_t(x,0)=\sin e^{2x}.
          \end{array}
    \right.
    \end{equation}
    \end{ex}

  Note that
  \begin{equation*}
    \begin{array}{c}
       u(x,0)=\cos e^{2x}=\sum\limits_{k\in \mathbb{N}} (-1)^k\frac{e^{4kx}}{(2k)!},\\
       u_t(x,0)=\sin e^{2x}=\sum\limits_{k\in \mathbb{N}_+} (-1)^{k-1}\frac{e^{2(2k-1)x}}{(2k-1)!}.
     \end{array}
  \end{equation*}
  So we have $u(x,t)\in FT(\mathbb{R}\oplus[0,\infty)),\ \{e^{2kx}\}_{k\in \mathbb{N}}$.  We can get the formal solution of Eq. \eqref{4ae} w.r.t. $\{e^{2kx}\}_{k\in \mathbb{N}}$:
  \begin{equation}\label{4aae}
    \begin{array}{r}
      u(x,t)=1+\sum\limits_{k\in  \mathbb{N}_+}\frac{(-1)^k\exp\left(2k(-at+2x)\right)}{(2k)!}\left(
      \cos2k\sqrt{-\Delta}t+\frac{a}{\sqrt{-\Delta}}\sin2k\sqrt{-\Delta}t\right)\\
      +\frac{(-1)^{k-1}\exp\left((2k-1)(-at+2x)\right)}{(2k-1)!(2k-1)\sqrt{-\Delta}}\sin(2k-1)\sqrt{-\Delta}t.
    \end{array}
    \end{equation}
  It's easy to prove that the series \eqref{4aae} is a solution of Eq. \eqref{4ae}.

     \begin{ex}
     \begin{equation}\label{oo}
        \left\{
          \begin{array}{ll}
           u_{t}-t(y-3)u_{xxy}=0,\  (x,y)\in \Omega\subseteq \mathbb{R}^2,\ t\geq 0, \\
           u(x,y,0)=x^3(x-\pi/2)^3\sin (y-3)^{\frac{3}{5}}.
          \end{array}
    \right.
    \end{equation}
    \end{ex}

    Note that
    \begin{equation*}
     x^3(x-\pi/2)^3\sin (y-3)^{\frac{3}{5}}=\sum\limits_{(k,m)\in \mathbb{N}_+^2}A_{km}(y-3)^{\frac{3(2m-1)}{5}}\sin 2kx,
    \end{equation*}
    where
    \begin{equation*}
        A_{km}= \frac{(-1)^{m-1}}{(2m-1)!}\frac{4}{\pi}\int_0^{\frac{\pi}{2}} x^3(x-\pi/2)^3\sin 2kx\text{d}x, \ \ \ \ (k,m)\in \mathbb{N}_+^2.
    \end{equation*}
    So $u(x,y,t)\in FT(\Omega\oplus [0,+\infty)),\ \left\{(y-3)^{\frac{3(2m-1)}{5}}\sin 2kx\right\}_{(k,m)\in \mathbb{N}^2_+}$. We can get the formal solution of Eq. \eqref{oo} w.r.t. $\left\{(y-3)^{\frac{3(2m-1)}{5}}\sin 2kx\right\}_{(k,m)\in \mathbb{N}^2_+}$:
    \begin{equation}\label{007}
         u(x,y,t)=\sum\limits_{(k,m)\in \mathbb{N}^2_+}A_{km}\exp\left(-\frac{6}{5}(2m-1) k^2 t^2\right)(y-3)^{\frac{3(2m-1)}{5}}\sin 2kx.
    \end{equation}
    Moreover, we can prove that
      \begin{equation*}
         \sum\limits_{k\in \mathbb{N}_+}\left|\frac{4k^2}{\pi}\int_0^{\frac{\pi}{2}} x^3(x-\pi/2)^3\sin 2kx\text{d}x\right|<+\infty,\ \ \ \ m\in \mathbb{N}_+.
    \end{equation*}
 So the series \eqref{007} is a solution of Eq. \eqref{oo}.

 \begin{ex}(Stokes Equations\cite{st0}-\cite{st2}).
 \begin{equation}\label{sss}
        \left\{
        \begin{array}{l}
           u_{jt}-\nu \sum\limits_{m=1}^{3}u_{jx_mx_m}+p_{x_j}=f_{j}(x,t),\ \ j=1,2,3, \\
           u_{1x_1}+u_{2x_2}+u_{3x_3}=0,\ \ \ \ \ \ t\geq 0,\  x=(x_1,x_2,x_3)\in \mathbb{R}^3, \\
            u_j(x,0)=\sum\limits_{k\in \mathbb{Z}^3}A_{jk}\varphi_{k},\ \  j=1,2,3,\\
            f_j(x,t)=\sum\limits_{k\in \in \mathbb{Z}^3}B_{jk}(t)\varphi_{k},\ \ j=1,2,3,
        \end{array}
        \right.
    \end{equation}
  where $\varphi_{k}=\exp(i\lambda_1k_1x_1+i\lambda_2k_2x_2+i\lambda_3k_3x_3)$, $\lambda_j\in \mathbb{R}\setminus\{0\},\ j=1,2,3,$ $\nu\geq 0$.
 \end{ex}

  Obviously $(u_1,u_2,u_3,p)^T\in FT(\mathbb{R}^3\oplus [0,+\infty)),\ \{\varphi_{k}\}_{k\in \mathbb{Z}^3}$.
  So we let
    \begin{equation}\label{s.4}
    \left\{
    \begin{array}{l}
      u_j(x,t)=\sum\limits_{k\in \mathbb{Z}^3}T_{jk}(t)\varphi_k,\ \ j=1,2,3;\\
      p(x,t)=\sum\limits_{k\in \mathbb{Z}^3}T_{4k}(t)\varphi_k.
    \end{array}
    \right.
    \end{equation}
   Suppose that the series \eqref{s.4} satisfy the following conditions:
    \begin{numcases}{}
         u_j=\sum\limits_{k\in \mathbb{Z}^3}T_{jk}(t)\varphi_k\in C(\mathbb{R}^3\oplus[0,+\infty)),\ \ \ \ j=1,2,3, \label{11nb}\\
        p=\sum\limits_{k\in \mathbb{Z}^3}T_{4k}(t)\varphi_{k}\in C(\mathbb{R}^3\oplus[0,+\infty)), \label{22nb}\\
        u_{jt}=\sum\limits_{k\in \mathbb{Z}^3}T'_{jk}(t)\varphi_{k}\in C(\mathbb{R}^3\oplus[0,+\infty)),\  \ \ \ j=1,2,3, \label{33nb}\\
         u_{jx_mx_m}=\sum\limits_{k\in \mathbb{Z}^3} -(\lambda_mk_m)^2T_{jk}(t)\varphi_{k}\in C(\mathbb{R}^3\oplus[0,+\infty)),\  \ \ \ m,j=1,2,3, \label{55nb}\\
         p_{x_j}=\sum\limits_{k\in \mathbb{Z}^3}i\lambda_jk_jT_{4k}(t)\varphi_{k}\in C(\mathbb{R}^3\oplus[0,+\infty)),\  \ \ \ j=1,2,3. \label{77nb}
    \end{numcases}
    By substituting the series \eqref{11nb}-\eqref{77nb} into Eq. \eqref{sss} we get
    \begin{equation*}
    \left\{
      \begin{array}{l}
        \sum\limits_{k\in \mathbb{Z}^3}[T'_{jk}+\sum\limits_{m=1}^{3}\nu (\lambda_mk_m)^2T_{jk}
    +i\lambda_jk_j T_{4k}-B_{jk}]\varphi_{k}=0,\ \
     j=1,2,3,\\
     \sum\limits_{k\in \mathbb{Z}^3}(i\lambda_1k_1 T_{1k}+i\lambda_2k_2 T_{2k}+i\lambda_3k_3 T_{3k})\varphi_{k}=0,\\
     u_j(x,0)=\sum\limits_{k\in \mathbb{Z}^3}A_{jk}\varphi_{k}=\sum\limits_{k\in \mathbb{Z}^3}T_{jk}(0)\varphi_k,\ \  j=1,2,3.
      \end{array}
      \right.
    \end{equation*}
    For any $k\in \mathbb{Z}^3$, we let
     \begin{equation}\label{sim2}
      \left\{
      \begin{array}{r}
         T'_{jk}+ \sum\limits_{m=1}^{3}\nu (\lambda_mk_m)^2T_{jk}+i\lambda_jk_j T_{4k}-B_{jk}=0,\ \ j=1,2,3,\\
        \lambda_1k_1 T_{1k}+\lambda_2k_2 T_{2k}+\lambda_3k_3 T_{3k}=0,\ \ \ \ \ \ \ \ \ \ \ \ \ \ \ \ \ \ \ \ \ \ \ \ \\
          T_{jk}(0)=A_{jk},\ \ \ \ \ \ \ \ \ \ \ \ \ \ \ \ \ \ \ \ \ \ \ \ j=1,2,3.
      \end{array}
      \right.
    \end{equation}
     For every $j=1,2,3$, the first equation in Eq. \eqref{sim2} is multiplied by $\lambda_jk_j$, then we can induce that
     \begin{equation*}
    T_{4k}\sum\limits_{j=1}^{3}i(\lambda_jk_j)^2-\sum\limits_{j=1}^{3}B_{jk}\lambda_jk_j=0,\ \  k\in \mathbb{Z}^3.
    \end{equation*}
   Hence we have
    \begin{equation*}\label{3s.3}
    \left\{
    \begin{array}{l}
    T_{4,(0,0,0)}=a,\ \ \ \ \ \ \ \  a\ \text{is an arbitrary constant,}\\
    T_{4k}=\frac{\sum\limits_{j=1}^{3}k_j\lambda_jB_{jk}(t)}{\sum\limits_{j=1}^{3}i(k_j\lambda_j)^2}, \ \ \ \ \ \ k\in \mathbb{Z}^3\setminus\{0\},\\
    T_{jk}=\exp\left( -\sum\limits_{m=1}^{3}\nu (k_m\lambda_m)^2t\right)
    \left(\int\limits_{0}^{t}\left(B_{jk}(s)-iT_{4k}(s)k_j\lambda_j\right)\exp\left( \sum\limits_{m=1}^{3}\nu (k_m\lambda_m)^2s\right)\text{d}s+A_{jk}\right),\\
    \ \ \ \ \ \ \ \ \ \ \  \ \ \ \ \ \ \ \ \ \ \  \ \ \ \ \ \ \ \ \ \ \ \ \ \  \ \ \ \ \ \ \ \ \ \ \ \ \ \ \ \ \ \ \ \  \ \ \ \ \ \ \ \ \ \ \ \
     \ \ \ \ \ \ \ \ \  \ \ \ \ \ \ \ \ \ \ \ \ \ \ \ \ \ \ \ \ \ \ \ \ \ \ \ \ \ \ \ \ \ \ \ \ \ \ \ \ \ \ \
   j=1,2,3, \ \  k\in \mathbb{Z}^3.
    \end{array}
    \right.
    \end{equation*}\vskip8pt

    Clearly we have:

     \begin{thm}\label{thmst1} If
    \begin{equation}\label{ieq22}
      \sum\limits_{k\in \mathbb{Z}^3}\ \sum\limits_{m=1}^{3}k_m^2\left(|B_{jk}(t)|+|A_{jk}|\right)<+\infty,\ \ \ \ t\geq 0,\ \ m=1,2,3,
    \end{equation}
    then the series \eqref{s.4} we obtain is a solution of Eq. \eqref{sss}.
    \end{thm}

    If $u_j(x,0),\ f_j(x,t)\ j=1,2,3$ are the real-valued functions, then we have $A_{jk}=\overline{A_{j,-k}}$, $B_{jk}(t)=\overline{B_{j,-k}(t)}$, $j=1,2,3,\ k\in \mathbb{Z}^3$. So we can induce that:

    \begin{thm}\label{thmst1} If $u_j(x,0),\ f_j(x,t)\ j=1,2,3$ are the real-valued functions, then so do the functions \eqref{s.4} we obtain.
    \end{thm}

\section{Series solutions to the cauchy problem for some more general LPDEs} \label{secl}

In this section, using an iterative method with respect to \eqref{1.bf}, we deal with several LPDEs. This technique can solve many LPDEs.

 \begin{ex}\label{exl}
     \begin{equation}\label{ma.1}
        \left\{
          \begin{array}{ll}
           u_{y}-u_{xy}-\left(e^{e^{-(x+2)}}-1\right)u=ye^{-(x+2)},\\
             (x,y)\in \Omega=\{(x,t)\mid x>0,\ 0\leq y\leq x\}, \\
           u(x,0)=1+e^{-(x+2)}.
          \end{array}
    \right.
    \end{equation}
     \end{ex}

    Clearly we have
    \begin{equation*}
      \exp(e^{-(x+2)})-1=\sum_{k\in \mathbb{N}_+}\frac{e^{-k(x+2)}}{k!}.
    \end{equation*}
     Next we set
    \begin{equation}\label{ma.2}
       u(x,y)=\sum\limits_{k\in \mathbb{N}}T_k(y)e^{-k(x+2)}.
    \end{equation}
    Suppose that the series \eqref{ma.2} satisfies  the following conditions:
    \begin{numcases}{}
         u=\sum\limits_{k\in \mathbb{N}}T_k(y)e^{-k(x+2)}\in C(\Omega),\label{zb1}\\
     u_t= \sum\limits_{k\in \mathbb{N}}T'_k(y)e^{-k(x+2)}\in C(\Omega),\label{zb2}\\
        u_{xt}= \sum\limits_{k\in \mathbb{N}_+}-kT'_k(y)e^{-k(x+2)}\in C(\Omega),\label{zb4}\\
        \left(e^{e^{-(x+2)}}-1\right)u= \sum\limits_{k\in \mathbb{N}_+}\ \sum\limits_{m=0}{k-1}\frac{T_m(y)}{(k-m)!}e^{-k(x+2)} \in C(\Omega).
        \label{zb5}
    \end{numcases}
    Substituting the series \eqref{ma.2} into the equations \eqref{ma.1} we have
    \begin{equation*}\label{ma.4}
\left\{
    \begin{array}{c}
     T'_0(y)+(2T'_1(y)-T_0(y)-t)e^{-(x+2)}
    +\sum\limits_{k=2}^{+\infty} \left[(k+1)T'_k(y)
    -\sum\limits_{m=0}^{k-1}\frac{1}{(k-m)!}T_m(y)\right]e^{-k(x+2)}=0,\\
 u(x,0)=\sum\limits_{k\in \mathbb{N}}T_k(0)e^{-k(x+2)}=1+e^{-(x+2)}.
    \end{array}
\right.
    \end{equation*}
    Let
   \begin{equation*}\label{ma.5}
   \left\{
    \begin{array}{r}
    T'_0(t)=0, \ \ \ \ \ \ \ \ \ \ \ \ \ \ \ \ \ \  \ \ \ \ \ \ T_0(0)=1, \ \ \ \ \ \  \ \ \ \ \  \\
     2T'_1(t)-T_0(t)-t=0,  \ \ \ \ \ \ \ \ \ \ \ \  \ \ \ \ \ \ T_1(0)=1, \ \ \ \ \ \  \ \ \ \ \  \\
    (k+1)T'_k(t)-\sum\limits_{m=0}^{k-1}\frac{1}{(k-m)!}T_m(t)=0, \ \ \ \ \ \ T_k(0)=0,\ \ k\geq 2.
    \end{array}
    \right.
    \end{equation*}
   Then we have
   \begin{equation*}\label{ma.6}
   T_k(t)=
   \left\{
    \begin{array}{r}
    1,\ \ \ \ \ \ \ \ \ \ \ \ \ \ \ \ \ \ \ \ \ \ \ \ k=0, \\
     \frac{1}{4}t^2+\frac{1}{2}t+1,\ \ \ \ \ \ \ \ \ \ \ \ \ \ \ \  k=1, \\
    \frac{1}{k+1}\int\limits_{0}^{t}\sum\limits_{m=0}^{k-1}\frac{1}{(k-m)!}T_m(s)\text{d}s,\ \ \ \ k\geq 2.
    \end{array}
    \right.
    \end{equation*}
    Next we prove that the formal solution \eqref{ma.2} is also a solution of \eqref{ma.1}.
     By the induction method, we can prove that
    \begin{equation*}\label{tt1}
     0<T_k(y)\leq e^{ky},\ \ \ \ k\in \mathbb{N}.
    \end{equation*}
    So we have
    \begin{equation*}
      0<T_k(y)e^{-k(x+2)}\leq e^{-k(x-y)-2k}\leq e^{-2k},\ \ \ \ (x,t)\in \Omega,\  k\in \mathbb{N}.
    \end{equation*}
    Hence the series \eqref{ma.2} converges uniformly on $\Omega$.
    It means that the formal solution \eqref{ma.2} satisfies \eqref{zb1}.  Moreover, we can prove that
    \begin{equation*}
    \left\{
    \begin{array}{r}
       |T'_k(y)e^{-k(x+2)}|=\frac{1}{k+1}\left|\sum_{m=0}^{k-1}\frac{1}{(k-m)!}T_m(y)\right|e^{-k(x+2)}\leq e^{-2k},\ \ \ \ k\geq 2,\\
      |-kT'_k(y)e^{-k(x+2)}|\leq ke^{-2k},\ \ \ \ \ \ \ \ \ \ \ \ \ \ \ \ \ \ \ \ \ \ \ \ \ \ \ \ k\geq 2,\\
      \left|\sum_{m=0}^{k-1}\frac{1}{(k-m)!}T_m(y)e^{-k(x+2)}\right|\leq ke^{-2k},\ \ \ \  \ \ \ \ \ \ \ \ \ \ \ \ \ \ \ \ \ k\geq 2.
    \end{array}
    \right.
    \end{equation*}
    So the series \eqref{ma.2} we obtain satisfies \eqref{zb2}-\eqref{zb5}.
    Therefore it is a solution of \eqref{ma.1}.

    \begin{ex}
 \begin{equation}\label{5551}
 \left\{
   \begin{array}{l}
      u_{t}+u+(x+3)^{\frac{2}{4}}u_{x}=0,\ \ \ \ x\geq0,\ t\geq 0,\\
      u(x,0)=\sin\ (x+3)^{-\frac{1}{4}}=\sum\limits_{k\in \mathbb{N}_+}\frac{(-1)^{k+1}(x+3)^{-\frac{2k-1}{4}}}{(2k-1)!}.
    \end{array}
    \right.
 \end{equation}
  \end{ex}

 Similar as Example \ref{exl}, we can get a solution:
\begin{equation*}
  u(x,t)=\sum\limits_{k\in \mathbb{N}_+}T_{k}(t)(x+3)^{-\frac{k}{4}},
\end{equation*}
where
\begin{equation*}
T_{k}(t)=\left\{
\begin{array}{l}
   e^{-t},\ \ \ \ \ \ \ \ k=1,\\
   \ 0,\ \ \ \ \ \ \ \ \ \ k=2,4,6,\cdots,\\
e^{-t}\left(\int_0^t\frac{k-2}{4} T_{k-2}(s)e^{s}\text{d}s +\frac{(-1)^{\frac{k-1}{2}}}{k!}\right),\ \ \ \ k=3,5,7,\cdots.
\end{array}
  \right.
\end{equation*}

\section{Series solutions to the cauchy problem for some NPDEs}

In this section, similar as Section \ref{secl}, using an iterative method with respect to \eqref{1.bf}, we deal with several NLPEs. This technique also can solve many NPDEs.

 \begin{lem}\label{idl} (Abel identities \cite{ci}) For every $k\in \mathbb{N}_+$, we have
      \begin{equation*}
        k(k+1)^k=\sum\limits_{m=1}^{k}\left(
                              \begin{array}{c}
                                k+1 \\
                                m
                              \end{array}
                            \right)
                            m^m(k+1-m)^{k-m},
      \end{equation*}
      where $\left(
                              \begin{array}{c}
                                k+1 \\
                                m
                              \end{array}
                            \right)=\frac{(k+1)!}{m!(k+1-m)!}$.
 \end{lem}

 \begin{ex}\label{exn1} (Inviscid Burgers' equation).
   \begin{equation}\label{1IB}
  \left\{
  \begin{array}{l}
     u_t+uu_x=0,\ \ (x,t)\in \Omega=\{(x,t)\mid t\geq 0,\ x\in [0,11]\},\\
   u(x,0)=1+e^{x-12}.
  \end{array}
  \right.
   \end{equation}
    \end{ex}

Next we let
\begin{equation}\label{IB1}
  u(x,t)=\sum\limits_{k\in \mathbb{N}}T_{k}(t)e^{k(x-12)}.
\end{equation}
Suppose that the following conditions hold:
 \begin{numcases}{}
 u=\sum\limits_{k\in \mathbb{N}}T_{k}(t)e^{k(x-12)}\in C(\Omega), \label{IB2}\\
  u_t=\sum\limits_{k\in \mathbb{N}}T'_{k}(t)e^{k(x-12)}\in C(\Omega), \label{IB3}\\
  u_x=\sum\limits_{k\in \mathbb{N}_+}k T_{k}(t)e^{k(x-12)}\in C(\Omega),\label{IB4}\\
 uu_x=\sum\limits_{k\in \mathbb{N}_+}\ \sum\limits_{r=1}^{k}r T_{r}(t)T_{k-r}(t)e^{k(x-12)}\in C(\Omega).\label{IB5}
    \end{numcases}
Substituting the series \eqref{IB1} into \eqref{1IB}, we get
\begin{equation*}\label{IB6}
\left\{
\begin{array}{c}
 T'_{0}+(T'_{1}+T_0T_1)e^{x-12}
 +\sum\limits_{k=2}^{+\infty}(T'_{k}+kT_0T_k+\sum\limits_{r=1}^{k-1}rT_{r}T_{k-r})e^{k(x-12)}=0,\\
 u(x,0)=\sum\limits_{k\in \mathbb{N}}T_{k}(0)e^{k(x-12)}=1+e^{x-12}
\end{array}
\right.
\end{equation*}
Note that the sequence $\{e^{k(x-12)}\}_{k\in \mathbb{N}}$ is linearly independent, so we have
\begin{equation*}
\left\{
\begin{array}{r}
    T'_{0}=0,\ \ \ \ \ \ \ \ \ \ \ \ \ \ \ \ \ \ \ \ \  \  \  \ T_{0}(0)=1,\ \ \ \ \ \ \ \ \  \ \\
    T'_{1}+T_0T_1=0,\ \ \ \ \ \ \ \ \ \ \ \ \ \  \  \ T_{1}(0)=1,\ \ \ \ \ \  \ \ \  \ \\
   T'_{k}+kT_0T_k+\sum\limits_{r=1}^{k-1}rT_{r}T_{k-r}=0, \ \ \  \ T_{k}(0)=0, \ \ k\geq 2.
\end{array}
  \right.
\end{equation*}
Then by Lemma \ref{idl}, we can get
\begin{equation*}
 T_{k}(t)=\left\{
\begin{array}{l}
    \ \ 1,\ \ \ \ \ \ \ \ \ k=0,\\
    e^{-t},\ \ \ \ \ \ \ \ k=1,\\
e^{-kt}\int\limits_0^t \sum\limits_{r=1}^{k-1}-rT_{r}(s)T_{k-r}(s)e^{ks}\text{d}s=(-1)^{k+1}\frac{k^{k-1}}{k!} t^{k-1}e^{-kt},\ \ \ \  k\geq 2.
\end{array}
  \right.
\end{equation*}
So we get
\begin{equation}\label{IBj}
   u(x,t)=1+e^{-t+x-12}+\sum\limits_{k\geq 2}(-1)^{k+1}\frac{k^{k-1}}{k!} t^{k-1}e^{k(-t+x-12)}.
\end{equation}\vskip8pt

Next we prove that the series \eqref{IBj} satisfies \eqref{IB2}-\eqref{IB5}.
Note that $\frac{k^m}{m!}t^m\leq e^{kt},\  m,k\in \mathbb{N}$, so we have
\begin{equation*}
 |T_k(t)e^{k(x-12)}|\leq  \frac{1}{k}e^{k(x-12)}\leq  \frac{1}{k}e^{-k},\ \ \ \ k\geq 2.
\end{equation*}
So the series \eqref{IBj} converges uniformly on $\Omega$.
    It means that the formal solution \eqref{IBj} satisfies \eqref{IB2}.
    Moreover, we can prove that
    \begin{equation*}
    \left\{
    \begin{array}{r}
    |\sum\limits_{r=1}^{k-1}rT_{r}T_{k-r}|=\frac{(k-1)k^{k-1}}{k!}t^{k-2}e^{-kt}=\frac{k^{k-2}}{(k-2)!}t^{k-2}e^{-kt}\leq 1,\ \ \ \ \ \ k\geq 2,\\
      |T'_k(t)e^{k(x-12)}|=|kT_0T_k+\sum\limits_{r=1}^{k-1}rT_{r}T_{k-r}|e^{k(x-12)}\leq 2e^{-k},\ \ \ \ \ \ k\geq 2,\\
       |kT_k(t)e^{k(x-12)}|\leq e^{-k},\ \ \ \ \ \ \ \ \ \ \ \ \ \ \ \ \ \ \ \  \ \ \ \ \ \ \ \ \ \ \ \ \ \ \ \ \ \ \ \ \ k\geq 2.
    \end{array}
    \right.
\end{equation*}
    Thus the series \eqref{IBj} is a solution of \eqref{1IB}.

    \begin{ex}
  \begin{equation}\label{qq0}
  \left\{
    \begin{array}{l}
  u_{t}+(x+1)^2 u_{xx}+u_{x}u=0,\ \ \ \ x\geq 1,\ t\geq 0,\\
  u(x,0)=(x+1)^{-1}+(x+1)^{-2}.
    \end{array}
  \right.
  \end{equation}\
   \end{ex}

Similar as Example \ref{exn1}, we let
\begin{equation}\label{qq2}
  u(x,t)=\sum\limits_{k\in \mathbb{N}_+}T_{k}(t)(x+1)^{-k},
\end{equation}
Suppose that the following conditions hold:
\begin{numcases}{}
   u=\sum\limits_{k\in \mathbb{N}_+}T_{k}(t)(x+1)^{-k}\in C([1,+\infty)\oplus [0,+\infty)), \label{1qq}\\
   u_x=\sum\limits_{k\in \mathbb{N}_+}-kT_{k}(t)(x+1)^{-k-1}\in C([1,+\infty)\oplus [0,+\infty)), \label{2qq}\\
   u_t=\sum\limits_{k\in \mathbb{N}_+}T'_{k}(t)(x+1)^{-k}\in C([1,+\infty)\oplus [0,+\infty)), \label{3qq}\\
   u_{xx}=\sum\limits_{k\in \mathbb{N}_+}k(k+1)T_{k}(t)(x+1)^{-k-2}\in C([1,+\infty)\oplus [0,+\infty)), \label{4qq}\\
   u_{x}u=\sum\limits_{k\geq 3}\ \sum\limits_{r=1}^{k-2}-r T_{r}(t)T_{k-1-r}(t)(x+1)^{-k}\in C([1,+\infty)\oplus [0,+\infty)).\label{5qq}
\end{numcases}
Substituting \eqref{qq2} into \eqref{qq0}, we get
\begin{equation*}
\left\{
\begin{array}{r}
 (T'_{1}+2T_1)(x+1)^{-1}+(T'_{2}+6T_2)(x+1)^{-2}+\sum\limits_{k\geq 3}(T'_{k}+k(k+1)T_k\\
\ \ \ \ \ \ \ -\sum\limits_{r=1}^{k-2}rT_{r}T_{k-1-r})(x+1)^{-k}=0,\\
  u(x,0)=\sum\limits_{k\in \mathbb{N}_+}T_{k}(0)(x+1)^{-k}=(x+1)^{-1}+(x+1)^{-2}.\ \ \ \ \ \ \ \ \ \
\end{array}
\right.
\end{equation*}
Note that the sequence $\{(x+1)^{-k}\}_{k\in \mathbb{N}_+}$ is linearly independent, so we have
\begin{equation*}
\left\{
\begin{array}{r}
    T'_{1}+2T_1=0,\ \ \ \ \ \ \ \ \ \ \ \ \ \ \ \ \ \ \ \ \ \ \ T_1(0)=1,\ \ \ \ \ \ \  \ \ \ \ \\
    T'_{2}+6T_2=0,\ \ \ \ \ \ \ \ \ \ \ \ \ \ \ \ \ \ \ \ \ \ \ T_2(0)=1,\ \ \ \ \ \ \  \ \ \ \ \\
   T'_{k}+k(k+1)T_k-\sum\limits_{r=1}^{k-2}rT_{r}T_{k-r}=0,\ \ \ \ \ \ \ \ T_k(0)=0,\ \ k\geq 3.
\end{array}
  \right.
\end{equation*}
Then we get
\begin{equation*}
 T_{k}(t)=\left\{
\begin{array}{c}
    e^{-2t},\ \ \ \ \ \ \ \ \ \ \ \ \ \ \ \ \  k=1,\\
    e^{-6t},\ \ \ \ \ \ \ \ \ \ \ \ \ \ \ \ \  k=2,\\
e^{-k(k+1)t}\int\limits_0^t \sum\limits_{r=1}^{k-2}r T_{r}(s)T_{k-1-r}(s)e^{k(k+1)s}\text{d}s ,\ \ \ \ k\geq 3.
\end{array}
  \right.
\end{equation*}
  Next we prove that the formal solution \eqref{qq2} satisfies \eqref{1qq}-\eqref{5qq}.
 By the induction method, we can prove that
  \begin{equation}\label{ttf}
    0<T_k(t)\leq e^{-(k+1)t},\ \ \ \ k\in \mathbb{N}_+.
 \end{equation}
  So the series $\sum_{k\in \mathbb{N}_+}T_{k}(t)(x+1)^{-k}$ converges uniformly on $[1,+\infty)\oplus [0,+\infty)$.
    It means that the formal solution \eqref{qq2} satisfies  \eqref{1qq}. Moreover, we can prove that
    \begin{equation*}
    \left\{
    \begin{array}{r}
     |-kT_{k}(t)(x+1)^{-k-1}|\leq ke^{-(k+1)t}(x+1)^{-k-1}\leq 2^{-k-1}k,\ \ \ \ \ \ \ \ \ \ \ \ \ \ \ \ \ k\geq 3,\\
      |T'_{k}(t)(x+1)^{-k}|=|k(k+1)T_k-\sum\limits_{r=1}^{k-2}rT_{r}T_{k-r}|(x+1)^{-k}\leq 2k(k+1)2^{-k},\ \ \ \ \ \ k\geq 3,\\
       |k(k+1)T_k(t)(x+1)^{-k-2}|\leq k(k+1)2^{-k-2},\ \  \ \ \ \ \ \ \ \ \ \ \ \ \ \ \ \ \ \ \ \ \ \ \ \ \ \ \  \ \ \ \ \ \ \ \ k\geq 3,\\
     |\sum\limits_{r=1}^{k-2}-r T_{r}(t)T_{k-1-r}(t)(x+1)^{-k}|\leq (k-2)(k-1) 2^{-k}, \ \ \ \ \ \ \ \ \ \ \ \ \ \ \ \ \ \ k\geq 3.
    \end{array}
    \right.
\end{equation*}
     So the formal solution \eqref{qq2} satisfies  \eqref{2qq}-\eqref{5qq}.
    Thus it is a solution of \eqref{qq0}.



\section*{Acknowledgments}

The paper is supported by the Natural Science Foundation of China (no. 11371185) and the Natural Science Foundation of Inner Mongolia,
China (no. 2013ZD01).


\end{document}